\documentclass[12pt]{article}
\usepackage{amsmath,amsthm,amssymb,hyperref, color}
\usepackage{fullpage}

\newtheorem{thm}{Theorem}[section]

\newtheorem{lem}[thm]{Lemma}
\newtheorem{define}[thm]{Definition}

\def\F{{\mathbb{F}}}

\def\Z{{\mathbb{Z}}}

\def\C{{\mathbb{C}}}

\def\P{{\mathbb{P}}}

\newcommand{\ip}[2]{\langle #1,#2 \rangle}

\def\_{\,\,\,\,\,}

\def\rank{\textsf{rank}}

\newcommand{\remove}[1]{}

\begin{document}

\title{Rank of incidence matrices over integers modulo a prime power}

\author{Zeev Dvir\thanks{Department of Mathematics  and Department of Computer Science,
Princeton University.
Email: \texttt{zdvir@princeton.edu}. Research supported by NSF grant DMS-2246682.}}

\date{}
\maketitle

\begin{abstract}
In this note we prove an upper bound on the $\F_p$-rank of the incidence matrix of points and hyperplanes in $(\Z/p^k \Z)^n$, improving a recent bound of Laba and Trainer when $k$ is large.
\end{abstract} 

%\pagenumbering{arabic}

%%%%%%%%%%%%%%%%%%%%%%%%%%%%%%%%%%%%%%%%%%%%%%%%%5
\section{Introduction}
%%%%%%%%%%%%%%%%%%%%%%%%%%%%%%%%%%%%%%%%%%%%%%%%%5

Let $\F_p$ be a finite field of prime order $p$. We denote by $R = \Z/p^k\Z$ the ring of integers modulo $p^k$. We define the projective space $\P (R)^{n-1}$ to be the set of vectors $v \in R^n$ with at least one unit coordinate modulo the relation $v \sim u$ given by scaling by a unit in $R$. We can pick unique representatives for $\P(R)^{n-1}$ in $R^n$ (for example, scaling the first unit coordinate to $1$) and so think of $\P(R)^{n-1}$ as a subset of $R^n$. For two vectors $v,u \in R^n$ we define the inner product $\ip{u}{v} = \sum_{i=1}^n v_iu_i \mod p^k$.  We denote the $\F_p$-rank of an integer matrix $M$ by  $\rank_p(M)$.

For $b \in \P(R)^{n-1}$ and $\lambda \in R$ we define the hyperplane
$$ H_{b,\lambda} = \{ v \in R^n \,|\, \ip{v}{b}=\lambda \} .$$
Let  $A(p^k,n)$ be the matrix whose rows are the indicator vectors of all hyperplanes $H_{b,\lambda}$. Formally, the rows of $A(p^k,n)$ are indexed by $\P(R)^{n-1} \times R$ and the columns are indexed by $R^n$ and the entry in position $( (b,\lambda), v )$ is equal to $1$ if $\ip{b}{v}=\lambda$ and is equal to $0$ otherwise. We will be interested in giving an upper bound on the $\F_p$-rank of this matrix. The first such non-trivial upper bound was given by Laba and Trainor \cite{LabaTrainor2025GeneralizedPolynomials}. They actually define a slightly different matrix $A^*(p^k,n)$ in which the rows are indexed by $\P(R)^{n-1} \times R^n$, the columns are indexed by $R^n$ and the entry in position $((b,a),v)$ is 1 iff $\ip{b}{v-a} = 0$. However, the ranks of the two matrices are identical as $A^*(p^k,n)$ has the same rows as $A(p^k,n)$ with each row repeated multiple times.

The rank bounds proven in \cite{LabaTrainor2025GeneralizedPolynomials} are summarized below
\begin{thm}[\cite{LabaTrainor2025GeneralizedPolynomials}]
Let $A(p^k,n)$ be defined above. Then
\[  \rank_p(A(p^k,n)) \leq \min\left\{ {p^k + n -1 \choose n}, (2n){\lfloor \frac{p^k}{2} \rfloor +(n-1)(p-1) +n \choose n} \right\} 
	\]
\end{thm}

It is instructive to  think of these bounds as having a leading term of the  form $C(p,k,n) \cdot p^{kn}$ (say, when $n$ is growing) as  the number of distinct hyperplanes $H_{b,\lambda}$ grows asymptotically like $\sim p^{kn}$. Both bounds above have $C(p,k,n)$ vanishing with  $n$ but independent of $k$. Our main result below shows an upper bound with $C(p,k,n)$ which goes to zero with $k$. The proof is via a connection established in \cite{DharDvir2021SquarefreeKakeya} between the rank of incidence matrices and the size of the smallest Kakeya sets in $R^n$. A Kakeya set is a set $S \subset R^n$ containing a line in every direction $b \in \P(R)^{n-1}$. In \cite{DharDvir2021SquarefreeKakeya} it was shown that the size of the smallest such $S$ is bounded from below by the rank of a matrix  $W^*(p^k,n)$ which is obtained from $A(p^k,n)$ by taking only rows corresponding to hyperplanes passing through the origin (i.e., whose index is $(b,0)$ for $b \in \P(R)^{n-1}$). We observe that one can tweak this reduction to prove the same result for $A(p^k,n)$ and then plug in the size of the smallest known Kakeya set in $R^n$ (due to Dhar \cite{Dhar2024KakeyaZN}) to prove the following theorem.

\begin{thm}\label{thm-main}
Let $p$ be prime. $k,n \geq 1$, $s \geq 0$ integers such that $k = \frac{p^{s+1}-1}{p-1}$. Then, 
$$ \rank_p( A(p^k,n) ) \leq \frac{p^{kn}}{k^{n-1}}(1 - 1/p)^{-n} $$
\end{thm}

We comment that one can in fact prove a bound for all $k$'s (not just of the given form) but this requires going into Dhar's construction and gives a somewhat cumbersome dependence on $k$.

For comparison, the best known lower bound on the $\F_p$-rank of $A(p^k,n)$ is obtained by reduction to the rank of yet another matrix $W(p^k,n)$ whose rank was originally lower bounded by Arsovski \cite{Arsovski2024padicKakeya} (see  \cite{Dhar2023BeyondPolynomialMethodThesis} for an improved bound\footnote{The rank bound appeared in an early draft of \cite{Arsovski2024padicKakeya} and does not appear in the final  version.}). The matrix $W(p^k,n)$ is the same as $W^*(p^k,n)$ (hyperplanes through the origin) but with all $b \in R^n$ (not just those with at least one unit coordinate).
\begin{thm}\cite{Arsovski2024padicKakeya,Dhar2023BeyondPolynomialMethodThesis}
Let $p$ be prime. $k,n \geq 1$ integers. Then
$$ \rank_p( W(p^k,n) )  \geq { \lceil \frac{p^k}{k} \rceil + n-1 \choose n-1} $$
\end{thm}

The relation between the ranks of $W(p^k,n)$ and $A(p^k,n)$ was established in \cite{LabaTrainor2025GeneralizedPolynomials}. 
\begin{thm}\cite{LabaTrainor2025GeneralizedPolynomials}
	Let $p$ be prime. $k,n \geq 1$ integers. Then
	$$ \rank_p( A(p^k,n) ) \geq \frac{1}{2k(k+1)} \cdot \rank_p( W(p^k,n+1) ) - 1 \geq \frac{1}{2k(k+1)}{ \lceil \frac{p^k}{k} \rceil + n \choose n} - 1. $$
\end{thm}

\section{Proof of Theorem~\ref{thm-main}}

\begin{define}
Let $\gamma$ be a complex primitive $p^k$th root of unity for prime $p$ and natural number $k$.
Let $M$ be a complex matrix with entries in $\{0,1,\gamma,\hdots,\gamma^{p^k-1}\}$. We define $ M|_{\gamma=1}$ to be an integer matrix of the same dimensions of $M$ and with entries
$$  (M|_{\gamma=1})_{ij} = \left\{
	\begin{array}{ll}
		0  & \mbox{if } M_{i,j}=0 \\
		1 & \mbox{otherwise} 
	\end{array}
\right. $$	
\end{define}

We will  need the following lemma from \cite{DharDvir2021SquarefreeKakeya}.
\begin{lem}\label{lem:rankCtoF}\cite{DharDvir2021SquarefreeKakeya}
Let $\gamma$ be a complex primitive $p^k$th root of unity for prime $p$ and natural number $k$. Let $M$ be a complex matrix with entries in $\{0,1,\gamma,\hdots,\gamma^{p^k-1}\}$. Then we have,
\em{$$\text{rank}_{\C}(M)\ge \text{rank}_{p}(M|_{\gamma=1}).$$}
\end{lem}

The following lemma is obtained by modifying the argument in \cite{DharDvir2021SquarefreeKakeya} slightly.
\begin{lem}\label{lem-rankreduce}
Suppose $S \subset R^n$ is a Kakeya set. Then
$$ \rank_p( A(p^k,n) ) \leq |S|. $$
\end{lem}
\begin{proof}
Let $\gamma$ be a complex primitive $p^k$th root of unity.
Let $F$ be a complex matrix	 whose rows are indexed by $S \subset R^n$ and columns indexed by $R^n$ with entries 
$$ F_{x,y} = \gamma^{\ip{x}{y}} ,$$
where we identify the set $R$ with the subset $\{0,1,\ldots,p^k-1\} \subset \Z$.

For each $b \in \P(R)^{n-1}$ let 
$$ L_b = \{ u_b + t b \,|\, t \in R \} \subset S$$
be a line in direction $b$ contained in $S$. For each $b \in \P(R)^{n-1}$ and $\lambda \in R$ we define a function $$\phi_{b,\lambda} : S \mapsto \C$$ as follows: For all $x \not\in L_b$ we set  $\phi_{b,\lambda}(x) = 0$ . If $x \in L_b$ then let $t \in R$ be such that $x = u_b + t b$ and set $$\phi_{b,\lambda}(u_b + t b)= \frac{1}{p^k}\gamma^{-\lambda t}.$$ Notice that $t$ is uniquely determined since $b$ has at least one unit coordinate. Finally, we let  $B$ be a complex matrix with rows indexed by $\P(R)^{n-1} \times R$ and columns indexed by $S$ whose rows are given by the various $\phi_{b,\lambda}$. More formally, the entry of $B$ in position $((b,\lambda),x)$ is $\phi_{b,\lambda}(x)$. 

We now define $$ M = B \cdot F $$ which trivially has $\rank_\C( M) \leq |S|.$  The entry of $M$ in position $((b,\lambda),y)$ is given by
\begin{eqnarray*}
	M_{(b,\lambda),y} &=& \sum_{x \in S}\phi_{b,\lambda}(x) \gamma^{\ip{x}{y}} \\
	&=&  \frac{1}{p^k}\sum_{t \in R}\gamma^{-\lambda t} \gamma^{\ip{u_b + tb}{y}}\\
   &=& \gamma^{\ip{u_b}{y}}\frac{1}{p^k}\sum_{t \in R}\gamma^{t (\ip{b}{y} - \lambda)} \\
   &=& \gamma^{\ip{u_b}{y}} \cdot 
   {\begin{cases}
1 &  \ip{b}{y}=\lambda \mod p^k \\
0  & \text{otherwise} 
\end{cases}}
\end{eqnarray*}
Hence, we see that $M|_{\gamma=1} = A(p^k,n)$ and so, by Lemma~\ref{lem:rankCtoF}, we have
$$ \rank_p( A(p^k,n) ) \leq \rank_{\C}( M) \leq |S| .$$
\end{proof}

Combining Lemma~\ref{lem-rankreduce} with the following construction of Dhar, completes the proof of Theorem~\ref{thm-main}. 

\begin{thm}[\cite{Dhar2024KakeyaZN}]
	Let $p$ be prime. $k,n \geq 1$ , $s \geq 0$ integers such that $k = \frac{p^{s+1} - 1}{p-1}$. Then, there exists a Kakeya set $S \subset R^n$ with 
	$$ |S| \leq \frac{p^{kn}}{k^{n-1}}(1 - 1/p)^{-n}. $$
	
\end{thm}

We comment that one can adapt the proof in \cite{Dhar2024KakeyaZN} to work for any $k$ (by padding the coefficient in the construction with zeros) but the final bound obtained is a bit cumbersome to state.

\bibliographystyle{alpha}

\bibliography{IncidenceRank}

\end{document}